\documentclass{birkjour}
\usepackage{subfigure}
\usepackage{verbatim}
 \newtheorem{thm}{Theorem}[section]
 \newtheorem{corollary}[thm]{Corollary}
 \newtheorem{lemma}[thm]{Lemma}
 \newtheorem{Proposition}[thm]{Proposition}
 \theoremstyle{definition}
 
 \theoremstyle{remark}
 \newtheorem{remark}[thm]{Remark}
 
 \numberwithin{equation}{section}

 \newcommand{\R}{\mathbb{R}}
  
   \newcommand{\Z}{\mathbb{Z}}

\renewcommand{\L}{\mathcal{L}}

\begin{document}

%
%

\title[Centroaffine Duality for Polygons in $3$-Space]
 {{Centroaffine Duality for Spatial Polygons}}

\author[M.Craizer]{Marcos Craizer}

\address{%
Departamento de Matem\'{a}tica- PUC-Rio\br
Rio de Janeiro, RJ, Brasil}
\email{craizer@puc-rio.br}

\author[S.Pesco]{Sinesio Pesco}

\address{%
Departamento de Matem\'{a}tica- PUC-Rio\br
Rio de Janeiro, RJ, Brasil}
\email{sinesio@puc-rio.br}



\thanks{The authors want to thank CNPq and CAPES for financial support during the preparation of this manuscript. \newline E-mail of the corresponding author: craizer@puc-rio.br}

\subjclass{ 53A15, 52C35}

\keywords{Affine vertices, Planar points, Flattening points, Support points, Four vertices theorem, Four planar points theorem.}

\date{November 10, 2018}

\begin{abstract}
In this paper, we discuss centroaffine geometry of polygons in $3$-space. For a polygon $X$ that 
is locally convex with respect to an origin together with a transversal vector field $U$, we define 
the centroaffine dual pair $(Y,V)$ similarly to \cite{Nomizu2}. We prove that vertices of $(X,U)$ 
correspond to flattening points for $(Y,V)$ and also that constant curvature polygons are dual to planar polygons. 
As an application, we give a new proof of a known $4$ flattening points theorem for spatial polygons.
\end{abstract}

\maketitle

\section{Introduction}

Affine differential geometry of curves in $3$-space studies differential concepts that are invariant under the affine group. When a distinguished 
origin is fixed, the study of such curves becomes part of the centroaffine differential geometry. 
We shall consider in this paper centroaffine concepts of spatial polygons that comes from discretizations of differential concepts,
and so our results can be classified into discrete differential centroaffine geometry. 
 
We shall consider spatial polygons $X(i)$ that are locally convex with respect to an origin $O\in\R^3$, which means that 
the determinant
\begin{equation*}\label{eq:alpha}
\left[ X(i-1)-O,X(i)-O,X(i+1)-O   \right]
\end{equation*}
does not changes sign, together with transversal vector fields $U(i)$, which means that 
\begin{equation*}\label{eq:beta}
\left[  X(i)-O,X(i+1)-O, U(i)  \right]
\end{equation*}
also does not change sign. The {\it (centroaffine) normal plane} is the plane generated by $X$ and $U$, while the
{\it (centroaffine) focal set} is the envelope of normal planes (\cite{Craizer-Pesco}). We show that the focal set
reduces to a line if and only if the pair $(X,U)$ has constant {\it curvature}.

We define the dual pair $(Y,V)$ of $(X,U)$, which is also a locally convex spatial polygon $Y$ together
with a transversal vector field $V$. This definition is a discrete counterpart of the centroaffine duality introduced in \cite{Nomizu2} for general
smooth codimension $2$ centroaffine immersions (see also \cite{Craizer-Garcia}, where the 
smooth spatial curves case is more explicit). 
We prove many properties of this duality, among them a correspondence between {\it vertices} of $(X,U)$ and {\it flattening points} of the dual pair $(Y,V)$. 
We show also that $(X,U)$ has constant curvature if and only if $(Y,V)$ is planar. As a consequence, we describe explicitly 
the constant curvature pairs. 

Equal-volume polygons are the discrete counterpart of curves parameterized by centroaffine arc-length (\cite{Craizer-Pesco}). 
Such polygons admit natural transversal vector fields that are parallel and unimodular. By using the duality tools, 
we give a characterization of the constant curvature equal-volume polygons. 

As an application of this centroaffine duality theory, we give another proof of a known $4$ flattening points theorem for polygons in $3$-space, by assuming that 
some radial projection is convex. This is a discrete counterpart of a theorem of Arnold for spatial smooth curves (\cite{Arnold}).

The paper is organized as follows: In section 2 we give the basic definitions and properties of centroaffine geometry of spatial polygons.
In section 3 we introduce the duality and prove its mais properties. In section 4 we characterize the constant curvature polygons, while
in section 5 we prove the above mentioned $4$ flattening points theorem. 

\section{Centroaffine Geometry of Polygons}

\subsection{Notation and terminology}

In this paper, we consider polygons in $3$-space. In order to avoid misunderstandings, 
we use the terms "edge" and "nodes" for the elements of a polygon, leaving 
the terms "vertex" and "flattening point" for special edges or nodes. 

We denote by $\Z$ the set of integers and by $\Z^*$ the set $\Z+\frac{1}{2}$. A discrete function is a map from $\Z$ or $\Z^*$ to $\R$, while a 
discrete vector field is a map from $\Z$ or $\Z^*$ to $\R^3$. For any discrete function $g(i)$, we shall use the notation 
$$
g'(i+\tfrac{1}{2})=g(i+1)-g(i),\ \ g''(i)=g'(i+\tfrac{1}{2})-g'(i-\tfrac{1}{2}), \ \
$$
and so on. For periodic functions of period $n$, we shall use the periodic notation $X(i+kn)=X(i)$, $k\in\Z$.
We shall denote by $\left[ \cdot, \cdot,\cdot\right]$ the standard volume in $\R^3$. 

\subsection{Locally convex polygons and transversal parallel vector fields}

Consider a closed spatial polygon $X$ with nodes $X(i)$, $1\leq i\leq n$, and $O\in\R^3$.
Let $\alpha=\alpha(X)$ be defined by 
\begin{equation*}\label{eq:alpha}
\alpha(i)=\left[ X(i-1)-O,X(i)-O,X(i+1)-O   \right].
\end{equation*}
We say that $X$ is {\it locally convex} with respect to $O$ if $\alpha(i)>0$, for any $i\in\Z$. 
Consider a vector field $U$ along $X$ given by $U(i)\in\R^3$, $1\leq i\leq n$. 
Define $\beta=\beta(X,U)$ by 
\begin{equation*}\label{eq:beta}
\beta(i+\tfrac{1}{2})=\left[  X(i)-O,X(i+1)-O, U(i)  \right].
\end{equation*}
We say that $U$ is {\it transversal} to $X$ with respect to $O$ if $\beta(i+\tfrac{1}{2})>0$, for any $i\in\Z$.
Along this paper, we shall consider $O$ to be the origin. 

A transversal vector field $U$ to $X$ is called {\it parallel} if 
\begin{equation}\label{eq:Defineb}
U'(i+\tfrac{1}{2})= - b(i+\tfrac{1}{2}) X'(i+\tfrac{1}{2}),
\end{equation}
for certain scalar function $b=b(X,U)$. We call $b(i+\tfrac{1}{2})$ the {\it curvature} of the edge $(i+\tfrac{1}{2})$.

\subsection{Affine focal sets and vertices}\label{sec:NormalLines}

For a locally convex spatial polygon $X$ with a transversal vector field $U$, the line
\begin{equation*}\label{eq:NormalLine}
X(i)+tU(i),\ \ t\in\R.
\end{equation*}
is called the {\it normal line} at $i$. The plane generated by $X$ and $U$ is called the {\it (centroaffine) normal plane}. 

\begin{lemma}\label{lemma:Ubar}
Consider a transversal vector field $\bar{U}$. Then $\bar{U}$ is a parallel vector field contained in the normal plane of the pair $(X,U)$ if and only if we can write
\begin{equation}\label{eq:UNormalPlane}
\bar{U}=cX+dU,
\end{equation}
for some constants $c$ and $d$, $d\neq 0$. 
\end{lemma}
\begin{proof}
If $\bar{U}$ is of the form \eqref{eq:UNormalPlane}, then clearly $\bar{U}$ is parallel and belongs to the normal plane. Conversely, 
if $\bar{U}$ belongs to the normal plane, then we can write $\bar{U}=cX+dU$. Since $\bar{U}$ is also parallel, $c$ and $d$ must be constants, thus
proving the lemma.
\end{proof}

The normal lines at $i$ and $i+1$ meet at
\begin{equation}\label{eq:NormalMeet}
E(i+\tfrac{1}{2})=X(i)+b^{-1}(i+\tfrac{1}{2})U(i)=X(i+1)+b^{-1}(i+\tfrac{1}{2})U(i+1).
\end{equation}

The {\it (centroaffine) focal set} is the envelope of the normal planes, i.e., the conical polyhedron with center $O$ over the polygon whose 
nodes are $E(i+\tfrac{1}{2})$, $i\in\Z$ (\cite{Craizer-Pesco}).

\begin{Proposition} \label{prop:ConstantCurvature}
The following statements are equivalent:
\begin{enumerate}
\item 
The focal set of $(X,U)$ reduces to a line.
\item
The curvature $b$ of $(X,U)$ is constant. 
\item
There exists a constant vector field $E$ contained in all normal planes of $(X,U)$. 
\item
There exists a constant vector field $E$ such that $E=cX+dU$, for some constants $c$ and $d$.
\end{enumerate}
\end{Proposition}
\begin{proof}
It is easy to verify that three consecutive normal lines at $i-1$, $i$ and $i+1$ of $(X,U)$ are concurrent if and only if $b(i+\tfrac{1}{2})=b(i-\tfrac{1}{2})$.
So $E$ defined by Equation \eqref{eq:NormalMeet} is constant if and only if $b$ is constant. Moreover, the affine focal set
reduces to a line if and only if $E$ is constant. Thus 
$(1)\Leftrightarrow(2)$ and the implication 
$(1)\Rightarrow(3)$ also holds. If we assume that (3) holds,  the affine focal set is the line passing through $O$ and $E$
and so (1) also holds. The equivalence between $(3)$ and $(4)$ is given by Lemma \ref{lemma:Ubar}. 
\end{proof}

We say that an edge $(i+\tfrac{1}{2})$ is a {\it vertex} of $(X,U)$ if 
\begin{equation}\label{eq:DefineVertex}
b'(i)\cdot b'(i+1)<0.
\end{equation}

\subsection{Flattening points}

Let $\Delta=\Delta(X)$ be defined by
\begin{equation*}
\Delta(i+\tfrac{1}{2})=\left[X(i+2)-X(i+1), X(i+1)-X(i),X(i)-X(i-1)\right],
\end{equation*}
A node $i$ of the polygon $X$ is said to be a {\it flattening point} if 
\begin{equation*}
\Delta(i-\tfrac{1}{2})\cdot\Delta(i+\tfrac{1}{2})< 0.
\end{equation*}
Geometrically, the flattening point condition means that
$X(i-2)$ and $X(i+2)$ are in the same side with respect to the plane defined by $X(i-1),X(i),X(i+1)$.
We remark that, in \cite[p.218]{Pak}, a flattening point is called a {\it support point}. 

For each $i$, take $\lambda(i)=\lambda(X,U)(i)$ such that $-\lambda(i)X(i)+U(i)$ belongs to the osculating plane, i.e., 
\begin{equation}\label{eq:Lambda}
\left[ X'(i+\tfrac{1}{2}), X''(i), -\lambda(i)X(i)+U(i) \right]=0.
\end{equation}

\begin{lemma}
We can write 
\begin{equation}\label{eq:Lambda2}
-\lambda(i)X(i)+U(i)=AX'(i-\tfrac{1}{2})+BX'(i+\tfrac{1}{2}),
\end{equation}
where $\alpha(i) A=-\beta(i+\tfrac{1}{2})$.
\end{lemma}
\begin{proof}
If follows from Equation \eqref{eq:Lambda} that Equation \eqref{eq:Lambda2} holds and
$$
A\left[ X'(i-\tfrac{1}{2}), X'(i+\tfrac{1}{2}), X(i) \right]= \left[ -\lambda(i)X(i)+U(i), X'(i+\tfrac{1}{2}), X(i) \right],
$$
thus proving the lemma.
\end{proof}

\begin{lemma}\label{lemma:Delta}
We have that 
$$
\beta(i+\tfrac{1}{2})\Delta(i+\tfrac{1}{2})=\lambda'(i+\tfrac{1}{2})\alpha(i)\alpha(i+1).
$$
\end{lemma}

\begin{proof}
Take the difference of Equation \ref{eq:Lambda} at $i$ and $i+1$ to obtain, after some manipulations, the following relation:
$$
\left[ X'(i+\tfrac{1}{2}), -X'''(i+\tfrac{1}{2}), -\lambda(i)X(i)+U(i) \right]=-\lambda'(i+\tfrac{1}{2})\left[ X'(i+\tfrac{1}{2}),X''(i+1),X(i)\right].
$$
Now from Equation \eqref{eq:Lambda2} we obtain
$$
\frac{\beta(i+\tfrac{1}{2})}{   \alpha(i) }\  \Delta(i+\tfrac{1}{2})= \lambda'(i+\tfrac{1}{2})\alpha(i+1),
$$
thus proving the lemma.
\end{proof} 

From the above lemma, we conclude the following proposition:

\begin{Proposition}\label{prop:CharFlat}
The node $i$ is a flattening point for $X$ if and only if 
$$
\lambda'(i-\frac{1}{2})\cdot \lambda'(i+\frac{1}{2})<0, 
$$
where  $\lambda=\lambda(X,U)$ and $U$ is any transversal parallel vector field along $X$. 
\end{Proposition}

\subsection{Equal-volume polygons} 

We say that the polygon $X$ is {\it equal-volume} with respect to the origin $O$ if 
$\alpha(i)=1$, for any $i\in\Z$ (see \cite{Craizer-Pesco}). We say that a transversal vector field $U$ is {\it unimodular} with respect to $O$ if $\beta(i+\tfrac{1}{2})=1$, for any $i\in\Z$.

For equal-volume polygons, there is a natural choice of a transversal parallel vector field, namely
\begin{equation}\label{eq:Uequalvolume}
U(i)= X''(i)+\lambda(i)X(i), \ \ i\in\Z.
\end{equation}
It is easy to see that $U$ is unimodular.

\begin{lemma}
The vector field $U$ defined by Equation \eqref{eq:Uequalvolume}  is parallel. 
\end{lemma}
\begin{proof}
We use the results of \cite{Craizer-Pesco}. Write
\begin{equation*}
\left\{
\begin{array}{c}
X'''(i+\tfrac{1}{2})=-\rho_2(i)X'(i+\tfrac{1}{2})+\tau(i+\tfrac{1}{2})X(i+1)\\
X'''(i+\tfrac{1}{2})=-\rho_1(i+1)X'(i+\tfrac{1}{2})+\tau(i+\tfrac{1}{2})X(i),
\end{array}
\right.
\end{equation*}
for some scalar functions $\rho_1$, $\rho_2$ and $\tau$ satisfying the compatibility equation
$$
\tau(i+\tfrac{1}{2})=\rho_2(i)-\rho_1(i+1).
$$
Defining $\bar{\lambda}$ by $-\bar{\lambda}'(i+\tfrac{1}{2})=\tau(i+\tfrac{1}{2})$, we have that
$$
U(i)=X''(i)+\bar{\lambda}(i)X(i)
$$
is parallel. Now condition \eqref{eq:Lambda} says that
$$
\lambda(i)-\bar{\lambda}(i)=0,
$$
thus implying that $\bar\lambda=\lambda$.
\end{proof}

\begin{Proposition}\label{prop:ParallelUnimodularEV}
Let $\bar{U}$  be transversal vector field that is parallel and unimodular.
Then 
\begin{equation*}
\bar{U}= U+cX,
\end{equation*}
where $U$ is defined by Equation \eqref{eq:Uequalvolume} and $c$ is some constant. 
\end{Proposition}

\begin{proof}
Since $\bar{U}$ is parallel, $\bar{U}(i)=cX(i)+dU(i)$, for certain constants $c$ and $d$. Since $\beta(i+\tfrac{1}{2})=1$, we conclude that 
$d=1$, thus proving the proposition.
\end{proof}

\section{Duality}

\subsection{Definition and properties}

The general notion of duality for codimension $2$ centroaffine immersions can be found in \cite[N9]{Nomizu} and \cite{Nomizu2}. For an explicit description of 
centroaffine duality of smooth spatial curves, see \cite{Craizer-Garcia}. 
We describe here a version of this duality for polygons in $3$-space.

Since $\left[X(i), X(i+1), U(i) \right]>0$, one can define a polygon $Y$ and a vector field $V$ along $Y$ uniquely by the following conditions:
\begin{equation*}\label{eq:Dual1}
Y(i+\tfrac{1}{2})\cdot X'(i+\tfrac{1}{2})=0;\  Y(i+\tfrac{1}{2})\cdot U(i)=1;\ Y(i+\tfrac{1}{2})\cdot X(i)=0,
\end{equation*}
and
\begin{equation*}\label{eq:Dual2}
V(i+\tfrac{1}{2})\cdot X'(i+\tfrac{1}{2})=0;\  V(i+\tfrac{1}{2})\cdot U(i)=0;\ V(i+\tfrac{1}{2})\cdot X(i)=1.
\end{equation*}
We say that $(Y,V)$ is the {\it dual} pair of $(X,U)$. The following lemma is straightforward:

\begin{lemma}\label{lemma:Dual}
Consider a locally convex polygon $X$ with a parallel transversal vector field $U$ and denote by $(Y,V)$ its dual pair. 
Then
\begin{equation*}\label{eq:Dual3}
Y(i+\tfrac{1}{2})\cdot U(i+1)=1;\ Y(i+\tfrac{1}{2})\cdot X(i+1)=0,
\end{equation*}
and 
\begin{equation*}\label{eq:Dual4}
V(i+\tfrac{1}{2})\cdot U(i+1)=0;\ V(i+\tfrac{1}{2})\cdot X(i+1)=1.
\end{equation*}
Moreover, $(X,U)$ is the dual pair of $(Y,V)$.
\end{lemma}

Recall that 
$$
\alpha(Y)(i+\tfrac{1}{2})=\left[Y(i-\tfrac{1}{2}), Y(i+\tfrac{1}{2}), Y(i+\tfrac{3}{2})\right]
$$
and 
$$
\beta(Y,V)(i)=\left[Y(i-\tfrac{1}{2}), Y(i+\tfrac{1}{2}), V(i+\tfrac{1}{2})\right].
$$

\begin{lemma}\label{lemma:BetaDual}
We have that 
$$
\beta(Y,V)(i)=\frac{\alpha(i)}{\beta(i-\tfrac{1}{2}) \beta(i+\tfrac{1}{2}) }, 
\ \ \alpha(Y)(i+\tfrac{1}{2})=\frac{\alpha(i)\alpha(i+1)}{\beta(i-\tfrac{1}{2})\beta(i+\tfrac{1}{2})\beta(i+\tfrac{3}{2})},
$$
where $\alpha=\alpha(X)$ and $\beta=\beta(X,U)$. 
\end{lemma}
\begin{proof}
Write
\begin{equation*}
\beta(i+\tfrac{1}{2}) Y(i+\tfrac{1}{2})=X(i)\times X(i+1),\ \ \beta(i-\tfrac{1}{2}) Y(i-\tfrac{1}{2})=X(i-1)\times X(i),
\end{equation*}
Taking the vector product of both equations we obtain
$$
\alpha(i)X(i)=\beta(i-\tfrac{1}{2})\beta(i+\tfrac{1}{2})Y(i-\tfrac{1}{2})\times Y(i+\tfrac{1}{2}).
$$
Now take the dot product with $V(i+\tfrac{1}{2})$ to obtain the first formula. By duality, we can write
$$
\alpha(Y)(i+\tfrac{1}{2})=\beta(Y,V)(i)\beta(Y,V)(i+1)\beta(i+\tfrac{1}{2}).
$$
Now use the first formula to obtain the second one.
\end{proof}

From the above lemma, we conclude that $Y$ is a locally convex polygon and that $V$ is a transversal vector field.

\begin{lemma}\label{lemma:Wparallel}
The transversal vector field $V$ is parallel and
$$
V'(i)=\lambda(i)Y'(i),
$$
where $\lambda=\lambda(X,U)$. We conclude that $b(Y,V)=-\lambda(X,U)$. 
\end{lemma}

\begin{proof}
Observe that $V'(i)$ is orthogonal to $U(i)$ and to $X(i)$ and the same occurs with $Y'(i)$. Thus we conclude that
$V'(i)$ is parallel to $Y'(i)$, and we write $V'(i)=c(i)Y'(i)$. We claim that $c=\lambda(X,U)$. 
In fact, substituting 
$$
\beta(i+\tfrac{1}{2}) Y(i+\tfrac{1}{2})=X(i)\times X(i+1),\ \ \beta(i+\tfrac{1}{2}) V(i+\tfrac{1}{2})=X'(i+\tfrac{1}{2})\times U(i),
$$
in Equation \eqref{eq:Lambda} we obtain
$$
\lambda(i)Y(i+\tfrac{1}{2})\cdot X''(i)-V(i+\tfrac{1}{2})\cdot X''(i)=0.
$$
Since $Y(i+\tfrac{1}{2})\cdot X'(i+\tfrac{1}{2})=0$, we conclude that 
$$
Y(i+\tfrac{1}{2})\cdot X''(i)+Y'(i)\cdot X'(i-\tfrac{1}{2})=0, 
$$
and the same holds for $V$. Thus
$$
\lambda(i)Y'(i)\cdot X'(i-\tfrac{1}{2})-V'(i)\cdot X'(i-\tfrac{1}{2})=0,
$$
which implies that $(c(i)-\lambda(i))Y'(i)\cdot X'(i-\tfrac{1}{2})=0$. Since 
$$
Y'(i)\cdot X'(i-\tfrac{1}{2})=Y(i+\tfrac{1}{2})\cdot X'(i-\tfrac{1}{2})=-\frac{\alpha(i)}{\beta(i+\tfrac{1}{2})}\neq 0,
$$
the lemma is proved.
\end{proof}

Next lemma shows that duality preserves the centroaffine normal plane. 

\begin{lemma}
Let $(Y,V)$ be the dual of $(X,U)$ and consider another transversal vector field $\bar{U}$ given by Equation \eqref{eq:UNormalPlane}. Then the 
dual pair of $(X,\bar{U})$ is given by 
\begin{equation*}\label{eq:DualNormalPlane}
\bar{Y}=d^{-1}Y,\ \ \bar{V}=-cd^{-1}Y+V.
\end{equation*}
\end{lemma}
\begin{proof}
Straightforward verifications.
\end{proof}

\subsection{The equal-volume case}

\begin{lemma}\label{lemma:DualUnimodular}
Assume $X$ is equal-volume and $U$ is a parallel and unimodular transversal vector field. Denoting by $(Y,V)$
the dual pair, we have that also $Y$ is equal-volume and $V$ is a parallel and unimodular transversal vector field.
\end{lemma}

\begin{proof}
This lemma is a direct consequence of Lemma \ref{lemma:BetaDual}.
\end{proof}

\subsection{Coplanarity and concurrency} 

\begin{Proposition}\label{prop:Coplanarity}
Four consecutive nodes of $X$ are coplanar if and only if 
the corresponding three normal lines of $(Y,V)$ meet at a point. 
\end{Proposition}

\begin{proof}
From Section \ref{sec:NormalLines}, we have that three consecutive normal lines of $(Y,V)$ at $(i-\tfrac{1}{2})$, $(i+\tfrac{1}{2})$ and $(i+\tfrac{3}{2})$, 
are concurrent if and only if $b(Y,V)(i)=b(Y,V)(i+1)$. 
By Lemma \ref{lemma:Wparallel},  this is equivalent to $\lambda(X,U)(i)=\lambda(X,U)(i+1)$. From Lemma \ref{lemma:Delta}, this is equivalent to
$\Delta(X,U)(i+\tfrac{1}{2})=0$, thus proving the proposition. 
\end{proof}

We say that the polygon $X$ is {\it generic} if no four consecutive nodes are coplanar.

\begin{corollary}
The polygon $X$ is generic if and only if $b(Y,V)'(i)\neq 0$, for any $i\in\Z$. 
\end{corollary}

\subsection{Vertex and flattening points}

The main result of the section is the following:

\begin{Proposition}\label{prop:DualFlatVertex}
Assume that $X$ is a generic polygon. Then the node $i$ is a flattening point of $(X,U)$ if and only if the edge $i$ is a vertex of $(Y,V)$.
\end{Proposition}

\begin{proof}
We have that  the edge $i$ of $(Y,V)$ is a vertex if and only if 
$$
b(Y,V)'(i-\frac{1}{2})\cdot b(Y,V)'(i+\frac{1}{2})<0.
$$
From Lemma \ref{lemma:Wparallel}, this is equivalent to
$$
\lambda(X,U)'(i-\frac{1}{2})\cdot \lambda(X,U)'(i+\frac{1}{2})<0.
$$
By Proposition \ref{prop:CharFlat}, this condition is equivalent to the node $i$ of $(X,U)$ being a flattening point.
\end{proof}

\section{Planar and Constant Curvature Polygons} \label{sec:AffineCylindricalPedal}

\subsection{Affine cylindrical pedal}

Consider a locally convex planar polygon $x(i)$ and let $u(i)\in\R^2$ be a transversal planar vector field that is {\it parallel}, i.e., we can write
\begin{equation*}
u'(i+\tfrac{1}{2})=-b(i+\tfrac{1}{2})x'(i+\tfrac{1}{2}), \ \  i\in\Z,
\end{equation*}
for some scalar function $b=b(x,u)$. The lines $x+tu$, $t\in\R$, are called the {\it normal lines} of the pair $(x,u)$.

The {\it lifting} of $(x,u)$ 
is the pair $(X,U)$ given by
$$
X(i)=(x(i),1),\ \ U(i)=(u(i),0).
$$
Observe that $U$ is parallel along $X$ and that $b(X,U)=b(x,u)$.

The {\it affine cylindrical pedal} of $(x,u)$ is defined by
$$
Y(i+\tfrac{1}{2})=\left(y(i+\tfrac{1}{2}), -y(i+\tfrac{1}{2})\cdot x(i) \right),
$$
where $y$ denotes the co-normal vector field of $(x,u)$, i.e., 
$$
y(i+\tfrac{1}{2})\cdot x'(i+\tfrac{1}{2})=0,\ \ y(i+\tfrac{1}{2})\cdot u(i)=1.
$$
It is easy to verify that the constant vector field $E=(0,0,1)$ is transversal to $Y$ and $(Y,E)$ is dual to $(X,U)$. By Proposition \ref{prop:ConstantCurvature}, 
$(Y,E)$ is a constant curvature pair.

The following proposition says that if, conversely, we start with a spatial polygon transversal to a constant vector field $E$, then it is necessarily
the affine cylindrical pedal of some planar pair $(x,u)$. 

\begin{Proposition}\label{prop:AffineCylindricalPedal}
Assume that $Y(i+\tfrac{1}{2})=(y(i+\tfrac{1}{2}),z(i+\tfrac{1}{2}))$ is a locally convex spatial polygon transversal to a constant vector field $E$. Then $Y$ is the affine cylindrical
pedal of some planar parallel pair $(x,u)$.
\end{Proposition}
\begin{proof}
Denote by $(X,U)$ the dual of $(Y,E)$. Since $(Y,E)$ is a constant curvature pair, Proposition \ref{prop:Coplanarity} 
implies that $X$ is planar. Moreover, since $U$ is orthogonal to $E$, it must belong to the same plane. 
Thus $(X,U)$ is the lifting of some planar pair $(x,u)$. 
\end{proof}

\begin{remark} Consider any locally convex polygon $Y$. Then it is locally transversal to a constant vector field $E$ that we may assume, by an affine change of coordinates,
to be $(0,0,1)$. Then Proposition \ref{prop:AffineCylindricalPedal} implies $Y$ is locally an affine cylindrical pedal. 
\end{remark}

\begin{corollary}
Consider a polygon $Y$ in $3$-space and a transversal parallel vector field $V$ such that the pair $(Y,V)$ has constant curvature. 
Then $Y$ is the affine cylindrical pedal of a planar polygon $x$ with a transversal parallel vector field $u$.
\end{corollary}

\begin{proof}
It follows from Proposition \ref{prop:ConstantCurvature} that $(Y,V)$ has constant curvature if and only if there exists 
a constant transversal vector field $E$ satisfying $V=cY+E$, for some constant $c$.
By Proposition \ref{prop:AffineCylindricalPedal}, $Y$ is the affine cylindrical pedal of some planar pair $(x,u)$. 
\end{proof}

\subsection{Constant curvature equal-volume polygons}

A planar polygon $x$ is called {\it equal-area} if 
\begin{equation*}
\left[ x(i)-x(i-1), x(i+1)-x(i)  \right]=1,
\end{equation*}
where $[\cdot,\cdot]$ denotes the area of two planar vectors (see \cite{Craizer-Teixeira-Alvim}). We say that a transversal vector field $u$ is {\it unimodular} if 
\begin{equation*}
\left[  x(i+1)-x(i), u(i)  \right]=1.
\end{equation*}
For equal-area polygons, it is easy to verify that
\begin{equation*}
u(i)=x''(i)
\end{equation*}
is the only transversal vector field that is parallel and unimodular. Observe that, in this case,
the lifting $X=(x,1)$ of $x$ is equal-volume and the transversal vector field $U=(u,0)$ is parallel and unimodular. 

The affine cylindrical pedal of $(x,u)$ is a pair $(Y,E)$, where $Y$ is a locally convex polygon and $E=(0,0,1)$. Moreover, by Lemma \ref{lemma:DualUnimodular},
the pair $(Y,E)$ is also equal-volume and unimodular. 

We remark that it is not always true that a locally convex polygon $Y$ is the affine cylindrical pedal of a pair $(x,u)$ with $x$ equal-area and $u$ unimodular. 
In fact, we have the following proposition:

\begin{Proposition}\label{prop:AffineCylindricalEV}
Consider an equal-volume polygon $Y$ in $3$-space. Then it is the cylindrical pedal of an equal-area polygon $x$ with $u$ unimodular
if and only if there exists a unimodular constant vector field $E$ transversal to $Y$.
\end{Proposition} 

\begin{proof}
If $(Y,E)$ is unimodular, then its dual $(X,U)$ is also unimodular. By Proposition \ref{prop:AffineCylindricalPedal}, $(X,U)$ is the 
lifting of a planar pair $(x,u)$, with $x$ equal-area and $u$ unimodular. 
\end{proof}

Next proposition gives a characterization of equal-volume polygons of constant curvature:

\begin{Proposition}
Consider an equal-volume polygon $Y$ in $3$-space. Then it has constant curvature if and only if it is the affine cylindrical pedal
of some equal-area planar polygon $x$. 
\end{Proposition}

\begin{proof}
Given an equal-volume polygon $Y$, denote by $V$ a parallel and unimodular transversal vector field. Then Proposition \ref{prop:ConstantCurvature}
says that the pair $(Y,V)$ has constant curvature if and only if there exists a constant vector field $E$ that is also transversal and unimodular, which by 
Proposition \ref{prop:ParallelUnimodularEV} must be of the form $V+cY$, for some constant $c$. By Proposition \ref{prop:AffineCylindricalEV}, this is equivalent
to $Y$ being the affine cylindrical pedal of some planar equal-area polygon $x$. 
\end{proof}

\section{ Application: A $4$ Flattening Points Theorem}

As an application of the centroaffine duality, we shall give a new proof of a $4$ flattening theorem for convex polygons in $3$-space. 
The proof is based on a $4$-vertex theorem for planar polygons described in \cite{Tabach}.

\subsection{Statement of the theorem}

We say that a polygon $X$ is said to be {\it weakly convex} if it lies in the surface 
of its convex hull (\cite[p.201]{Pak}). We shall consider a stronger convexity condition, namely, that some radial projection of the spatial polygon is a planar convex polygon 
(see Figure \ref{fig:Fig1}). 
The following theorem is proved in \cite{Pak} with the hypothesis of weak convexity.

\begin{thm}\label{thm:DiscreteArnold}
Let $X$ be a generic closed polygon in $3$-space such that, for some center $O$, the radial projection of $X$ in a plane is a convex planar polygon.
Then $X$ admits at least $4$ flattening points. 
\end{thm}

\begin{figure}[htb]
 \centering
 \includegraphics[width=0.40\linewidth]{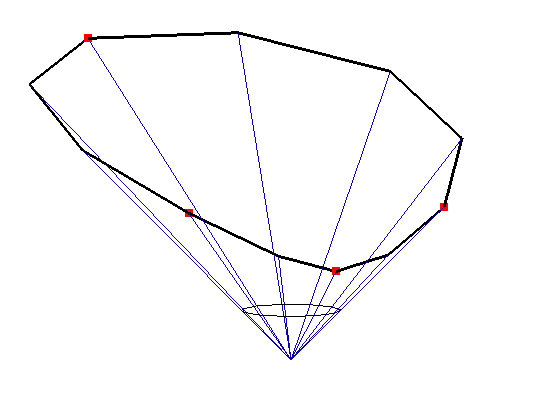}
 \caption{ A spatial polygon whose radial projection is convex and its flattening points. }
\label{fig:Fig1}
\end{figure}

This theorem is a polygonal version of the following well-known Arnold's $4$-flattening points theorem for smooth spatial curves (\cite{Arnold}). Recall that
a flattening point of a smooth curve $c:[a,b]\to\R^3$ is a point $t\in[a,b]$ such that $c'''(t)$ belongs to the osculating plane of $c$ at $t$. 

\begin{thm}\label{thm:Arnold}
Let $c:[a,b]\to\R^3$ be a closed smooth curve such that, for some center $O\in\R^3$, the radial projection of $c$ in a plane is a convex planar curve.
Then $c$ admits at least $4$ flattening points. 
\end{thm}

For a discussion of different types of convexity of spatial curves and other smooth $4$ flattening points theorems, see \cite{Uribe-Vargas}.

\subsection{Dual of a convex affine cylindrical pedal}

Assume that $X(i)=(x(i),z(i))$ is a locally convex spatial polygon transversal to $E=(0,0,1)$. By Proposition \ref{prop:AffineCylindricalPedal},  $X$ is the affine cylindrical pedal of a planar polygon $y(i+\tfrac{1}{2})$.

\begin{Proposition}\label{prop:ConvexAffineCylindricalPedal}
If $X(i)=\lambda(i)\left(\gamma(i),1\right)$, with  $\lambda(i)>0$ and $\gamma(i)$ convex containing $(0,0)$ in its interior, then $y(i+\tfrac{1}{2})$ is convex (see Figure \ref{fig:Fig2}).
\end{Proposition}

\begin{proof}
Recall that a locally convex planar polygon $y$ is convex if and only if its index is $1$. One can think of the index of a planar 
locally convex polygon as the sum of its external angles divided by $2\pi$.

If the index of $y$ were greater than $1$, 
then the index of its co-normal vector field $x$ would also be greater than $1$. On the other hand, by the convexity of $\gamma$, 
the polygon $x$ intersects each ray from $(0,0)$ at most once, which is a contradiction. 
\end{proof}

\begin{figure}[htb]
 \centering
 \includegraphics[width=0.60\linewidth]{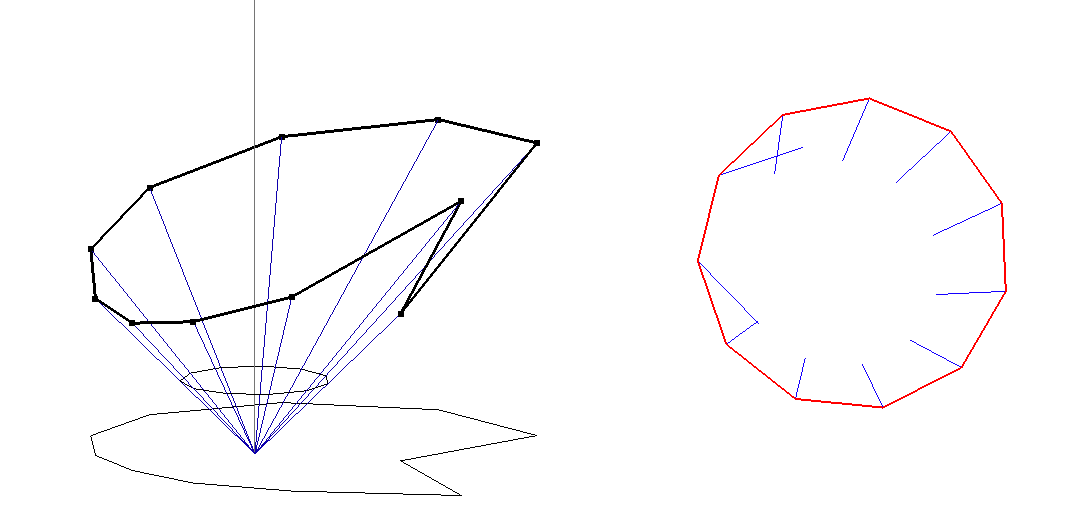}
 \caption{ The polygon $X(i)=(x(i),z(i))=\lambda(i)(\gamma(i),1)$ and its dual pair $(y,v)(i+\tfrac{1}{2})$. }
\label{fig:Fig2}
\end{figure}

\subsection{Proof of Theorem \ref{thm:DiscreteArnold} }

Consider a locally convex polygon $X$ in $3$-space whose projection in a plane with respect to a center $O$ is a convex planar curve.
We may assume that $O=(0,0,0)\in\R^3$ and that the plane of projection is $z=1$. Thus there exist $\lambda(i)>0$ such that
$$
X(i)=\lambda(i)(\gamma(i),1),
$$
where $\gamma(i)$ is a convex planar polygon. We can also assume, w.l.o.g., that $(0,0)$ is contained in the interior of $\gamma$.
This implies that $E=(0,0,1)$ is a transversal vector field along $X$.

Denoting by $(Y,V)$ the dual pair of $(X,E)$, Proposition \ref{prop:ConvexAffineCylindricalPedal} says that we can write $Y=(y,1)$ and $V=(v,0)$, for some planar convex polygon $y$ and parallel planar vector field $v$ along $y$. Moreover, $X=(x,z)$ is the affine pedal of $(y,v)$, i.e., $x$ is the co-normal of $(y,v)$ and
$$
z(i)=x(i)\cdot y(i+\tfrac{1}{2}).
$$

We claim that vertices of $(Y,V)$ correspond to vertices of $(y,v)$ in the sense of \cite{Tabach}. In fact, we can write
\begin{equation}\label{eq:b2}
v'(i)= -b(i)y'(i),
\end{equation}
where $b=b(Y,V)$ is defined in Equation \eqref{eq:Defineb}. This equation implies that $v$ is {\it exact} with respect to $y$  
and that the edge $(i)$ is a vertex of $(y,v)$ if and only if $b'(i-\tfrac{1}{2})\cdot b'(i+\tfrac{1}{2})<0$ (\cite{Tabach}). By Equation \eqref{eq:DefineVertex}, this 
is equivalent to the edge $(i)$ being a vertex of $(Y,V)$.

The vector field $v$ is called generic for $y$ in \cite{Tabach} if no $3$ consecutive normal lines $y(i+\tfrac{1}{2})+tv(i+\tfrac{1}{2})$ intersect at a point. This is equivalent to say that 
no $3$ consecutive lines $Y(i+\tfrac{1}{2})+tV(i+\tfrac{1}{2})$ intersect at a point. By Proposition \ref{prop:Coplanarity}, this is equivalent to the condition that no $4$ consecutive points of $X$ are coplanar, which in fact means that $X$ is generic. 

We are now in position to use the following theorem, proved in \cite{Tabach}:

\begin{thm}
Assume that $y$ is a planar convex polygon and that the generic transversal vector field $v$ is exact. Then the pair $(y,v)$ admits at least $4$ vertices. 
\end{thm}

From this theorem, we conclude that $(y,v)$, and hence $(Y,V)$, admits at least $4$ vertices.
By Proposition \ref{prop:DualFlatVertex}, this implies that $X$ admits at least $4$ flattenings, thus completing the proof of Theorem \ref{thm:DiscreteArnold}.

\end{document}